\pdfoutput=1
\documentclass[11pt]{amsart}
\usepackage{amsmath,amssymb,amsthm}
\usepackage[margin=1.1in]{geometry}
\usepackage{hyperref}
\usepackage{booktabs}

\newtheorem{theorem}{Theorem}
\newtheorem{lemma}[theorem]{Lemma}
\newtheorem{proposition}[theorem]{Proposition}
\newtheorem{corollary}[theorem]{Corollary}
\theoremstyle{definition}

\newtheorem{conjecture}[theorem]{Conjecture}
\theoremstyle{remark}
\newtheorem{remark}[theorem]{Remark}

\newcommand{\fl}[1]{\left\lfloor #1\right\rfloor}
\newcommand{\fr}[1]{\left\{ #1\right\}}
\newcommand{\dist}[1]{\left\| #1\right\|}
\newcommand{\N}{\mathbb{N}}
\newcommand{\rows}{\mathcal{M}}

\title[Kimberling's array $\lfloor k\varphi^n\rfloor$]{On three conjectures of Kimberling\\ concerning the array $\lfloor k\varphi^n\rfloor$}

\author{Alex Ashburn}
\address{Independent Researcher}
\email{aashburn1998@gmail.com}

\subjclass[2020]{11B83, 11B39, 11J71}
\keywords{Beatty sequence, golden ratio, Fibonacci numbers, Lucas numbers, Skolem--Bang theorem, OEIS}

\begin{document}

\begin{abstract}
Let $\varphi$ be the golden ratio and let $R_n=\{\lfloor k\varphi^n\rfloor : k\ge 1\}$ be the $n$-th row of the array $T(n,k)=\lfloor k\varphi^n\rfloor$ (OEIS A128440).
In 2022 Kimberling conjectured that the rows $R_{2n-1}$ and $R_{2n}$ are disjoint, and that after the two rows are merged and each entry is replaced by its rank, they become the lower and upper Wythoff sequences.
He also conjectured (OEIS A358359) that if $a(N)$ is the number of rows containing $N$, then every positive integer occurs infinitely often among the values of $a$.
We show that the first two conjectures follow quickly from the Skolem--Bang theorem, which also yields the exact rule for when two rows are disjoint: $R_i\cap R_j=\emptyset$ ($i<j$) if and only if $j-i$ is odd and divides $i$.
We then prove the third conjecture.
The main tools are an explicit determination of the rows containing an odd-indexed Lucas number, which extends a result of Noppakaew, Kanwarunyu and Wanitchatchawan, and a ``Lucas shift'' lemma: if $N+1$ is not of the form $L_{2e}$ with $e\ge1$, then adding a sufficiently large even-indexed Lucas number to $N$ does not change the set of rows containing it.
We also show that each value of $a$ is taken on a set of positive natural density, and we report computations up to $10^8$ suggesting that the least $N$ lying in exactly $v\ge 2$ rows is the Lucas number $L_{4v-5}$.
\end{abstract}

\maketitle

\section{Introduction}

Throughout, $\varphi=(1+\sqrt5)/2$, $\psi=-1/\varphi=(1-\sqrt5)/2$, and $F_n$, $L_n$ denote the Fibonacci and Lucas numbers ($F_0=0$, $F_1=1$, $L_0=2$, $L_1=1$).
For a real $x$ we write $\fr{x}=x-\fl{x}$ and $\dist{x}$ for the distance from $x$ to the nearest integer.
For irrational $\alpha>1$ let $B(\alpha)=\{\fl{k\alpha}:k\ge1\}$ be the Beatty sequence of $\alpha$.

Kimberling's array A128440 of the OEIS~\cite{oeis} is
\[
T(n,k)=\fl{k\varphi^n}\qquad(n,k\ge1),
\]
so its $n$-th row is $R_n:=B(\varphi^n)$.
Rows $1$ and $2$ are the lower and upper Wythoff sequences A000201 and A001950, which partition $\N=\{1,2,3,\dots\}$ by Beatty's theorem~\cite{beatty}, since $\varphi^{-1}+\varphi^{-2}=1$.
In November 2022 Kimberling recorded the following two conjectures in A128440.

\begin{conjecture}[Kimberling]\label{conj:disjoint}
For every $n\ge1$, the rows $R_{2n-1}$ and $R_{2n}$ are disjoint.
\end{conjecture}

\begin{conjecture}[Kimberling, stronger form]\label{conj:strong}
For every $n\ge1$, if the numbers in rows $2n-1$ and $2n$ are arranged together in increasing order and each number is replaced by its position in this ordering, then the two resulting rows are rows $1$ and $2$.
\end{conjecture}

In the companion entry A358359 he defined
\[
a(N)=\#\{m\ge1: N\in R_m\},
\]
the number of occurrences of $N$ in A128440, and conjectured the following.

\begin{conjecture}[Kimberling]\label{conj:occ}
Every positive integer occurs infinitely many times in the sequence $(a(N))_{N\ge1}$.
\end{conjecture}

The OEIS entries (as of October 2026) list these statements as conjectures. We could not find proofs of them stated in the literature, although Conjecture~\ref{conj:disjoint} follows at once from a theorem of Russo and Schwiebert~\cite{russo} (see below).
In August and September 2026 Kimberling added analogous arrays $\fl{kx^n}$ for the constants $x=1+\sqrt2$, $1+\sqrt3$, $\sqrt5-1$, $2+\sqrt2$, $e$ and $\pi$ (A399246, A399248, A399252, A399254, A399256, A399258), together with companion $0/1$ arrays recording which entries lie in the first row (see the guide in A399245), so the topic is of current interest.

Our results are as follows.
In Section~\ref{sec:disjoint} we observe that Conjectures~\ref{conj:disjoint} and~\ref{conj:strong} follow quickly from the classical Skolem--Bang theorem~\cite{skolem,bang}, and we obtain the complete rule for disjointness (Theorem~\ref{thm:pairs}):
\[
R_i\cap R_j=\emptyset\ (i<j)\iff j-i \text{ is odd and } (j-i)\mid i.
\]
In particular \emph{all} pairs of consecutive rows are disjoint, not just the pairs $(2n-1,2n)$, and the stronger rank statement also holds for all consecutive pairs (Corollary~\ref{cor:kimberling}).
We claim no novelty for this part beyond pointing out the connection and working out the arithmetic; in particular, as noted above, Conjecture~\ref{conj:disjoint} follows at once from~\cite{russo} (see the proof of Corollary~\ref{cor:kimberling}).
The same argument applies to $\fl{kx^n}$ whenever $x>1$ is a quadratic unit of norm $-1$, for example $x=1+\sqrt2$ (Theorem~\ref{thm:units}).

Section~\ref{sec:occ} contains the main new result, a proof of Conjecture~\ref{conj:occ} (Theorem~\ref{thm:occ}); one of its tools, the description of the rows containing an odd-indexed Lucas number (Proposition~\ref{prop:lucas}), extends a result of~\cite{nkw}.
Section~\ref{sec:numerics} reports computations, shows that each value of $a$ is taken on a set of positive natural density (Proposition~\ref{prop:density}), and states an open conjecture on first occurrences.

\section{Preliminaries}

We use the following standard facts (see e.g.~\cite{vajda}).
For $n\ge1$,
\begin{align}
\varphi^n&=F_n\varphi+F_{n-1}, \label{eq:phipow}\\
\varphi^n+\psi^n&=L_n, \label{eq:lucas}\\
\varphi^{-n}&=(-1)^n\,(F_{n+1}-F_n\varphi), \label{eq:phineg}
\end{align}
where~\eqref{eq:phineg} follows from $\varphi^{-n}=(-\psi)^n$ and $\psi^n=F_n\psi+F_{n-1}=F_{n+1}-F_n\varphi$.
From~\eqref{eq:lucas}, since $\psi^n=(-1)^n\varphi^{-n}$,
\begin{equation}\label{eq:lucas2}
\varphi^{n}=L_n-(-1)^n\varphi^{-n}\qquad(n\ge0).
\end{equation}
We also use d'Ocagne's identity $F_mF_{n+1}-F_{m+1}F_n=(-1)^nF_{m-n}$, the addition formula $F_{m+n}=F_mF_{n+1}+F_{m-1}F_n$, and the strong divisibility property $\gcd(F_m,F_n)=F_{\gcd(m,n)}$.

\begin{lemma}[Membership criterion]\label{lem:member}
Let $\alpha>1$ be irrational and $N\ge1$ an integer. Then $N\in B(\alpha)$ if and only if $\fr{N/\alpha}>1-1/\alpha$.
Moreover, if $\alpha\ge N+1$ then $N\notin B(\alpha)$.
\end{lemma}

\begin{proof}
$N=\fl{k\alpha}$ if and only if $N/\alpha<k<(N+1)/\alpha$.
The open interval $(N/\alpha,(N+1)/\alpha)$ has length $1/\alpha<1$ and irrational endpoints, so it contains an integer exactly when the least integer above $N/\alpha$, which is $N/\alpha+1-\fr{N/\alpha}$, is less than $(N+1)/\alpha$, that is, when $\fr{N/\alpha}>1-1/\alpha$.
If $\alpha\ge N+1$ then $\fl{k\alpha}\ge\fl{\alpha}\ge N+1$ for all $k\ge1$.
\end{proof}

\begin{lemma}\label{lem:dioph}
For every integer $q\ge1$ we have $\dist{q\varphi}>1/(3q)$.
Consequently, for integers $K\ge1$ and $i\ge1$, $\dist{K\varphi^{-i}}>1/(3KF_i)$.
\end{lemma}

\begin{proof}
Let $p$ be the integer nearest to $q\varphi$, so $p\ge1$.
Then $(p-q\varphi)(p-q\psi)=p^2-pq-q^2$ is an integer, and it is nonzero because $p/q$ is rational while $\varphi,\psi$ are not.
Hence $\dist{q\varphi}=|p-q\varphi|\ge 1/|p-q\psi|=1/(p+q/\varphi)$.
Since $p\le q\varphi+\tfrac12$, we get $p+q/\varphi\le q(\varphi+\varphi^{-1})+\tfrac12=q\sqrt5+\tfrac12<3q$.
For the second statement, by~\eqref{eq:phineg} we have $K\varphi^{-i}\equiv \pm KF_i\varphi\pmod 1$, so $\dist{K\varphi^{-i}}=\dist{KF_i\varphi}>1/(3KF_i)$.
\end{proof}

By Lemma~\ref{lem:member} with $\alpha=\varphi^m$,
\begin{equation}\label{eq:rowcrit}
N\in R_m\iff \fr{N\varphi^{-m}}>1-\varphi^{-m},
\end{equation}
and $N\notin R_m$ whenever $\varphi^m\ge N+1$.
In particular $a(N)$ is finite for every $N$.

\section{Disjointness of rows}\label{sec:disjoint}

\begin{theorem}[Skolem--Bang~\cite{skolem,bang}; see also~\cite{ayat}]\label{thm:SB}
Let $\alpha,\beta>1$ be irrational. Then $B(\alpha)\cap B(\beta)=\emptyset$ if and only if there are positive integers $r,s$ with $r/\alpha+s/\beta=1$.
\end{theorem}

The ``if'' direction is elementary.
If $r/\alpha+s/\beta=1$, then $B(\alpha/r)$ and $B(\beta/s)$ are complementary by Beatty's theorem.
Since $\fl{k\alpha}=\fl{(kr)(\alpha/r)}$, we have $B(\alpha)\subseteq B(\alpha/r)$, and likewise $B(\beta)\subseteq B(\beta/s)$.
The ``only if'' direction is deeper. Both directions are in Skolem's paper~\cite{skolem} (see~\cite[p.~26]{fraenkel}). Bang~\cite{bang} proved the companion criterion that, for irrational $\gamma,\beta>1$, $B(\beta)\subseteq B(\gamma)$ if and only if $r(1-1/\gamma)+s/\beta=1$ for some positive integers $r,s$. Applied with $\gamma=\alpha/(\alpha-1)$, whose Beatty sequence is the complement of $B(\alpha)$ by Beatty's theorem, it gives the ``only if'' direction again, since $1-1/\gamma=1/\alpha$.

\begin{theorem}\label{thm:pairs}
Let $1\le i<j$ and $d=j-i$. Then $R_i\cap R_j=\emptyset$ if and only if $d$ is odd and $d\mid i$.
\end{theorem}

\begin{proof}
By Theorem~\ref{thm:SB}, $R_i\cap R_j=\emptyset$ if and only if $r\varphi^{-i}+s\varphi^{-j}=1$ for some positive integers $r,s$, that is, $r\varphi^{d}+s=\varphi^{j}$.
By~\eqref{eq:phipow} and the linear independence of $1,\varphi$ over $\mathbb{Q}$, this is equivalent to
\[
rF_d=F_j,\qquad rF_{d-1}+s=F_{j-1}.
\]
Thus $r=F_j/F_d$, and by d'Ocagne's identity with $(m,n)=(j-1,d-1)$,
\[
s=\frac{F_dF_{j-1}-F_jF_{d-1}}{F_d}=\frac{(-1)^{d-1}F_{j-d}}{F_d}=(-1)^{d-1}\frac{F_i}{F_d}.
\]
So positive integers $r,s$ exist if and only if $F_d\mid F_j$, $F_d\mid F_i$, and $d$ is odd.
Since $F_j=F_{i+d}=F_iF_{d+1}+F_{i-1}F_d$, the condition $F_d\mid F_i$ already implies $F_d\mid F_j$.
Finally, for odd $d$ we have $F_d\mid F_i\iff d\mid i$.
This is trivial for $d=1$.
For $d\ge3$, strong divisibility gives $F_d\mid F_i\iff F_{\gcd(d,i)}=F_d$, and since $F_g<F_d$ for $1\le g<d$ when $d\ge3$, this holds if and only if $\gcd(d,i)=d$.
\end{proof}

For example, $R_3\cap R_6=\emptyset$, while $13=\fl{5\varphi^2}=\fl{2\varphi^4}\in R_2\cap R_4$.

\begin{corollary}\label{cor:kimberling}
For every $m\ge1$:
\begin{enumerate}
\item $R_m\cap R_{m+1}=\emptyset$. In particular Conjecture~\ref{conj:disjoint} holds.
\item For all $k,j\ge1$, $\fl{k\varphi^m}<\fl{j\varphi^{m+1}}$ if and only if $\fl{k\varphi}<\fl{j\varphi^2}$.
Consequently, when $R_m\cup R_{m+1}$ is listed in increasing order, the element $\fl{k\varphi^m}$ has position $\fl{k\varphi}$ and the element $\fl{j\varphi^{m+1}}$ has position $\fl{j\varphi^2}$. In particular Conjecture~\ref{conj:strong} holds.
\end{enumerate}
\end{corollary}

\begin{proof}
(1) is the case $d=1$ of Theorem~\ref{thm:pairs}.
(Directly: $F_{m+1}/\varphi^m+F_m/\varphi^{m+1}=(F_{m+1}\varphi+F_m)/\varphi^{m+1}=1$ by~\eqref{eq:phipow}, so $B(\varphi^m/F_{m+1})$ and $B(\varphi^{m+1}/F_m)$ are complementary, which is Theorem~2.2 of Russo and Schwiebert~\cite{russo}; these two sets contain $R_m$ and $R_{m+1}$, respectively.)

(2) If $k\varphi^m<j\varphi^{m+1}$ then $\fl{k\varphi^m}\le\fl{j\varphi^{m+1}}$, and equality is impossible by (1).
The same holds with the inequality reversed, and $k\varphi^m\neq j\varphi^{m+1}$ since $\varphi$ is irrational.
Hence $\fl{k\varphi^m}<\fl{j\varphi^{m+1}}\iff k<j\varphi\iff k\varphi<j\varphi^2\iff\fl{k\varphi}<\fl{j\varphi^2}$, where the last step uses the $m=1$ case of the same argument.
Now the position of $\fl{k\varphi^m}$ in $R_m\cup R_{m+1}$ is
\[
k+\#\{j\ge1: j\varphi<k\}=k+\fl{k/\varphi}=\fl{k+k/\varphi}=\fl{k\varphi},
\]
using $\varphi=1+1/\varphi$.
Likewise the position of $\fl{j\varphi^{m+1}}$ is $j+\#\{k\ge1:k<j\varphi\}=j+\fl{j\varphi}=\fl{j\varphi^2}$.
\end{proof}

\subsection*{Other quadratic units}
Besides Theorem~\ref{thm:SB}, the argument above uses only~\eqref{eq:phipow}, the irrationality of $\varphi$, d'Ocagne's identity, the addition formula, strong divisibility and the inequality $F_g<F_d$ for $1\le g<d$, $d\ge3$, so it extends to other units.
Let $c\ge1$ be an integer, let $x=(c+\sqrt{c^2+4})/2$ be the root greater than $1$ of $x^2=cx+1$, and let $P_0=0$, $P_1=1$, $P_{n+1}=cP_n+P_{n-1}$.
Then $x$ is irrational (since $c^2+4$ is not a perfect square), $x^n=P_nx+P_{n-1}$, d'Ocagne's identity $P_mP_{n+1}-P_{m+1}P_n=(-1)^nP_{m-n}$ and the addition formula $P_{m+n}=P_mP_{n+1}+P_{m-1}P_n$ hold, and $(P_n)$ is a strong divisibility sequence with $P_g<P_d$ for $1\le g<d$, $d\ge3$ (see e.g.~\cite{ribenboim}).
The proof of Theorem~\ref{thm:pairs} then gives the following.

\begin{theorem}\label{thm:units}
With $x$ as above and $1\le i<j$, the sequences $B(x^i)$ and $B(x^j)$ are disjoint if and only if $j-i$ is odd and $(j-i)\mid i$.
In particular $B(x^m)\cap B(x^{m+1})=\emptyset$ for every $m\ge1$.
\end{theorem}

For $c=2$ ($x=1+\sqrt2$) this describes the disjoint rows of Kimberling's array A399246.

\section{Occurrence counts: proof of Conjecture~\ref{conj:occ}}\label{sec:occ}

For $N\ge1$ let $\rows(N)=\{m\ge1:N\in R_m\}$, so $a(N)=\#\rows(N)$.
By Beatty's theorem $R_1\sqcup R_2=\N$, so $a(N)\ge1$.
By Corollary~\ref{cor:kimberling}, $\rows(N)$ contains no two consecutive integers.
By Lemma~\ref{lem:member}, $\rows(N)\subseteq\{m:\varphi^m<N+1\}$.
Hence $1\le a(N)\le\lceil \log_\varphi(N+1)/2\rceil$.

\begin{proposition}[Rows of odd-indexed Lucas numbers]\label{prop:lucas}
Let $m\ge3$ be odd. Then
\[
\rows(L_m)=\{\,i \text{ odd}: 1\le i\le (m-1)/2\,\}\cup\{m\},
\]
and hence $a(L_m)=\fl{(m+1)/4}+1$. In particular $a(L_{4v-5})=v$ for every $v\ge2$.
\end{proposition}

\begin{proof}
Since $m$ is odd, $L_m=\varphi^m-\varphi^{-m}$ by~\eqref{eq:lucas}, so $L_m\varphi^{-i}=\varphi^{m-i}-\varphi^{-m-i}$ for every $i\ge1$.
We apply~\eqref{eq:rowcrit} to each $i$.

\emph{Case $i=m$.} $L_m\varphi^{-m}=1-\varphi^{-2m}$, whose fractional part $1-\varphi^{-2m}$ exceeds $1-\varphi^{-m}$. So $m\in\rows(L_m)$.

\emph{Case $i>m$.} $0<L_m\varphi^{-i}<\varphi^{-1}$. Since $i\ge4$, we have $1-\varphi^{-i}>1-\varphi^{-2}=\varphi^{-1}$, so $i\notin\rows(L_m)$.

\emph{Case $i<m$, $i$ even.} Then $m-i$ is odd, and~\eqref{eq:lucas2} gives $\varphi^{m-i}=L_{m-i}+\varphi^{i-m}$. So $\fr{L_m\varphi^{-i}}=\varphi^{i-m}-\varphi^{-m-i}$, which lies in $(0,1)$ and is less than $\varphi^{-1}$. Since $1-\varphi^{-i}\ge 1-\varphi^{-2}=\varphi^{-1}$, we get $i\notin\rows(L_m)$.

\emph{Case $i<m$, $i$ odd.} Then $m-i\ge2$ is even, and $\varphi^{m-i}=L_{m-i}-\varphi^{i-m}$. So
\[
\fr{L_m\varphi^{-i}}=1-\varphi^{i-m}-\varphi^{-m-i},
\]
because $0<\varphi^{i-m}+\varphi^{-m-i}\le\varphi^{-2}+\varphi^{-4}<1$.
By~\eqref{eq:rowcrit}, $i\in\rows(L_m)$ if and only if $\varphi^{i-m}+\varphi^{-m-i}<\varphi^{-i}$, that is, $\varphi^{2i-m}+\varphi^{-m}<1$.
Here $2i-m$ is odd.
If $2i-m\le-1$, then $\varphi^{2i-m}+\varphi^{-m}\le\varphi^{-1}+\varphi^{-3}<1$.
If $2i-m\ge1$, then $\varphi^{2i-m}>1$.
So $i\in\rows(L_m)$ if and only if $i\le(m-1)/2$.

The number of odd $i$ with $1\le i\le (m-1)/2$ is $\fl{(m+1)/4}$.
For $m=4v-5$ this is $v-1$, so $a(L_{4v-5})=v$.
\end{proof}

For the rows $i\le(m-3)/2$, Proposition~\ref{prop:lucas} is the golden-ratio case $c=1$ of a theorem of Noppakaew, Kanwarunyu and Wanitchatchawan~\cite[Theorem~4.3]{nkw} on powers of the metallic means $(c+\sqrt{c^2+4})/2$ (their $k$ is our $c$): if $n\ge2i+2$, then $L_n\in R_i$ if and only if $n\equiv i\pmod 2$ (see also~\cite[Proposition~4.2]{nkw} for $i=1,2$).
The new part of Proposition~\ref{prop:lucas} is the treatment of the rows $i\ge(m-1)/2$, which is needed for the exact value of $a(L_m)$.

\begin{lemma}[Lucas shift]\label{lem:shift}
Let $N\ge1$ be such that $N+1\ne L_{2e}$ for every $e\ge1$.
Put $R=\fl{\log_\varphi(N+1)}+3$ and let $b$ be an even integer with
\begin{equation}\label{eq:bcond}
b\ge 2R\qquad\text{and}\qquad 6(N+1)\varphi^{2R-1-b}<1.
\end{equation}
Then $\rows(N+L_b)=\rows(N)$.
\end{lemma}

\begin{proof}
Write $M=N+L_b$.
Since $b$ is even, $L_b=\varphi^b+\varphi^{-b}$, so
\begin{equation}\label{eq:Mexp}
M\varphi^{-i}=N\varphi^{-i}+\varphi^{b-i}+\varphi^{-b-i}\qquad(i\ge1).
\end{equation}
For $1\le i\le b$, identity~\eqref{eq:lucas2} with $n=b-i$ and $(-1)^{b-i}=(-1)^i$ gives $\varphi^{b-i}=L_{b-i}-(-1)^i\varphi^{i-b}$. Hence
\begin{equation}\label{eq:delta}
M\varphi^{-i}\equiv N\varphi^{-i}+\delta_i\pmod 1,\qquad \delta_i:=-(-1)^i\varphi^{i-b}+\varphi^{-b-i},
\end{equation}
with $|\delta_i|<2\varphi^{i-b}$.
We compare $\rows(M)$ with $\rows(N)$ in three ranges of $i$.

\emph{(A) $1\le i\le R$.}
By~\eqref{eq:rowcrit}, $N\in R_i$ iff the point $N\varphi^{-i}$ of $\mathbb{R}/\mathbb{Z}$ lies in the open arc $(-\varphi^{-i},0)$, and $M\in R_i$ iff $N\varphi^{-i}+\delta_i$ does.
The distances from $N\varphi^{-i}$ to the endpoints $0$ and $-\varphi^{-i}$ of this arc are $\dist{N\varphi^{-i}}$ and $\dist{(N+1)\varphi^{-i}}$.
By Lemma~\ref{lem:dioph}, both exceed $1/(3(N+1)F_i)$.
Since $F_i\le\varphi^{i-1}$ and $i\le R$, condition~\eqref{eq:bcond} gives
\[
|\delta_i|<2\varphi^{i-b}\le\frac{1}{3(N+1)F_i}\cdot 6(N+1)\varphi^{2i-1-b}<\frac{1}{3(N+1)F_i}.
\]
So adding $\delta_i$ does not move the point across an endpoint of the arc, and $i\in\rows(M)\iff i\in\rows(N)$.

\emph{(B) $R<i\le b$.}
Here $\varphi^i\ge\varphi^{R+1}>\varphi^{\log_\varphi(N+1)+3}=\varphi^3(N+1)>N+1$, so $i\notin\rows(N)$, and moreover $(N+1)\varphi^{-i}<\varphi^{-3}$.
Put $y=N\varphi^{-i}+\delta_i$; by~\eqref{eq:delta}, $\fr{M\varphi^{-i}}=\fr{y}$.

If $i$ is odd then $i\le b-1$, so $0<y=N\varphi^{-i}+\varphi^{i-b}+\varphi^{-b-i}$ with $\varphi^{i-b}\le\varphi^{-1}$.
Adding $\varphi^{-i}$ gives $y+\varphi^{-i}\le(N+1)\varphi^{-i}+\varphi^{-1}+\varphi^{-b-i}<\varphi^{-3}+\varphi^{-1}+\varphi^{-4}=1$, so $\fr{y}=y<1-\varphi^{-i}$ and $i\notin\rows(M)$.

If $i$ is even then $y=N\varphi^{-i}-\varphi^{i-b}+\varphi^{-b-i}$.
If $y\ge0$, then $y+\varphi^{-i}\le (N+1)\varphi^{-i}+\varphi^{-b-i}<\varphi^{-3}+\varphi^{-4}<1$, so $i\notin\rows(M)$.
If $y<0$, then $y>-\varphi^{i-b}\ge-1$, so $\fr{y}=1+y$, so $i\in\rows(M)$ iff $y>-\varphi^{-i}$.
Multiplying the two conditions $y<0$ and $y>-\varphi^{-i}$ by $\varphi^i$ shows that $i\in\rows(M)$ iff
\begin{equation}\label{eq:window}
N+\varphi^{-b}<\varphi^{2i-b}<N+1+\varphi^{-b}.
\end{equation}
Put $2e=2i-b$, an even integer with $2e\le b$.
If $e\le0$, then $\varphi^{2e}\le1\le N$, so~\eqref{eq:window} fails.
If $e\ge1$, then $\varphi^{2e}=L_{2e}-\varphi^{-2e}$ by~\eqref{eq:lucas2}, and $\varphi^{-2e}\ge\varphi^{-b}$.
The left inequality of~\eqref{eq:window} gives $L_{2e}>N+\varphi^{-b}+\varphi^{-2e}>N$, so $L_{2e}\ge N+1$.
The right inequality gives $L_{2e}<N+1+\varphi^{-b}+\varphi^{-2e}\le N+1+2\varphi^{-2}<N+2$, so $L_{2e}\le N+1$.
Thus~\eqref{eq:window} forces $L_{2e}=N+1$, which is excluded by hypothesis.
So $i\notin\rows(M)$.

\emph{(C) $i>b$.}
By~\eqref{eq:Mexp}, $0<M\varphi^{-i}\le N\varphi^{-b-1}+\varphi^{-1}+\varphi^{-2b-1}$.
Since $b\ge2R\ge R+3$ we have $(N+1)\varphi^{-b-1}<\varphi^{-3}$, and hence $M\varphi^{-i}+\varphi^{-i}\le(N+1)\varphi^{-b-1}+\varphi^{-1}+\varphi^{-2b-1}<\varphi^{-3}+\varphi^{-1}+\varphi^{-4}=1$.
So $\fr{M\varphi^{-i}}=M\varphi^{-i}<1-\varphi^{-i}$ and $i\notin\rows(M)$.
Also $i\notin\rows(N)$, since $\varphi^i>\varphi^R>N+1$.

Combining (A), (B) and (C) gives $\rows(M)=\rows(N)$.
\end{proof}

\begin{remark}
The proof of case (B) shows what happens when the hypothesis on $N+1$ is dropped.
If $N+1=L_{2e}$ with $e\ge1$, then $\rows(N+L_b)=\rows(N)\cup\{b/2+e\}$ when $b/2+e$ is even, and $\rows(N+L_b)=\rows(N)$ otherwise.
We checked this numerically for all $N<1500$ and eight admissible values of $b$ each.
\end{remark}

\begin{theorem}\label{thm:occ}
For every $v\ge1$ there are infinitely many $N\ge1$ with $a(N)=v$. Hence Conjecture~\ref{conj:occ} is true.
\end{theorem}

\begin{proof}
For $v=1$ take $N_1=1$.
Since $1=\fl{\varphi}$ and $\varphi^m\ge\varphi^2>2$ for $m\ge2$, Lemma~\ref{lem:member} gives $\rows(1)=\{1\}$.
Also $N_1+1=2=L_0$ is not of the form $L_{2e}$ with $e\ge1$.
For $v\ge2$ take $N_v=L_{4v-5}$.
Then $a(N_v)=v$ by Proposition~\ref{prop:lucas}.
Moreover $N_v+1$ is not a Lucas number: for $m\ge3$ we have $L_m<L_m+1<L_m+L_{m-1}=L_{m+1}$ because $L_{m-1}\ge3$.
In both cases Lemma~\ref{lem:shift} applies to $N_v$, and the numbers $N_v+L_b$, for all sufficiently large even $b$, form an infinite set of integers with exactly $v$ occurrences.
\end{proof}

For example, Lemma~\ref{lem:shift} guarantees $a(4+L_b)=2$ for even $b\ge20$ and $a(29+L_b)=3$ for even $b\ge30$. Direct computation shows that in fact $a(4+L_b)=2$ for all even $b\ge10$ and $a(29+L_b)=3$ for all even $b\ge22$ (checked for $b<120$).

\section{Densities, computations and an open question}\label{sec:numerics}

Using exact integer arithmetic ($\fl{k\varphi^m}=\bigl\lfloor\bigl(kL_m+\lfloor\sqrt{5k^2F_m^2}\rfloor\bigr)/2\bigr\rfloor$) we computed $a(N)$ for all $N\le10^8$.
Table~\ref{tab:dens} lists the proportion of $N\le 10^8$ with each value of $a(N)$, together with the least such $N$.

\begin{table}[h]
\centering
\begin{tabular}{rrrl}
\toprule
$v$ & $\#\{N\le10^8: a(N)=v\}$ & proportion & least $N$ with $a(N)=v$\\
\midrule
1 & 47\,265\,868 & 0.47266 & 1 $=L_1$\\
2 & 44\,629\,641 & 0.44630 & 4 $=L_3$\\
3 & 7\,269\,082 & 0.07269 & 29 $=L_7$\\
4 & 724\,952 & 0.00725 & 199 $=L_{11}$\\
5 & 94\,426 & 0.00094 & 1\,364 $=L_{15}$\\
6 & 13\,623 & $1.4\cdot10^{-4}$ & 9\,349 $=L_{19}$\\
7 & 2\,031 & $2.0\cdot10^{-5}$ & 64\,079 $=L_{23}$\\
8 & 314 & $3.1\cdot10^{-6}$ & 439\,204 $=L_{27}$\\
9 & 53 & $5.3\cdot10^{-7}$ & 3\,010\,349 $=L_{31}$\\
10 & 10 & $1.0\cdot10^{-7}$ & 20\,633\,239 $=L_{35}$\\
\bottomrule
\end{tabular}
\caption{Distribution of $a(N)$ for $N\le 10^8$.}\label{tab:dens}
\end{table}

These data suggest that each value of $a$ is taken on a set of positive density.
This follows from Weyl's equidistribution theorem~\cite{weyl} together with Lemma~\ref{lem:dioph}.
For $m\ge1$ let
\[
S_m=\{x\in[0,1):\fr{(-1)^{m+1}F_mx}>1-\varphi^{-m}\},
\]
a union of $F_m$ open intervals of total length $\varphi^{-m}$, which is invariant under $x\mapsto x+1/F_m$ modulo~$1$.
Let $A(x)=\#\{m\ge1:x\in S_m\}$ and $A_M(x)=\#\{m\le M:x\in S_m\}$, and let $\lambda$ denote Lebesgue measure.

\begin{proposition}\label{prop:density}
For every $v\ge1$, the set $\{N:a(N)=v\}$ has natural density $d_v=\lambda(\{x\in[0,1):A(x)=v\})$, and $d_v>0$.
\end{proposition}

\begin{proof}
By~\eqref{eq:phineg}, $N\varphi^{-m}\equiv(-1)^{m+1}F_mN\varphi\pmod 1$.
So by~\eqref{eq:rowcrit}, $N\in R_m$ if and only if $x_N:=\fr{N\varphi}\in S_m$, and hence $a(N)=A(x_N)$.

\emph{Existence.} The set $\{x:A_M(x)=v\}$ is a finite union of intervals, so by Weyl's theorem the set $\{N:A_M(x_N)=v\}$ has density $\lambda(\{A_M=v\})$.
If $a(N)\ne A_M(x_N)$ then $N\in R_m$ for some $m>M$, and $\#(R_m\cap[1,X])\le(X+1)\varphi^{-m}$, so
\[
\#\{N\le X:a(N)\ne A_M(x_N)\}\le(X+1)\sum_{m>M}\varphi^{-m}=(X+1)\varphi^{1-M}.
\]
Similarly $\lambda(\{A_M\ne A\})\le\sum_{m>M}\lambda(S_m)=\varphi^{1-M}$.
Hence the upper and lower densities of $\{N:a(N)=v\}$ differ from $d_v$ by at most $2\varphi^{1-M}$, for every $M$.

\emph{Positivity.} Fix $N\ge2$ with $a(N)=v$: take $N=2$ if $v=1$ (here $\rows(2)=\{2\}$), and $N=L_{4v-5}$ if $v\ge2$ (Proposition~\ref{prop:lucas}).
Put $R=\fl{\log_\varphi(N+1)}+3$ and $\delta=1/(3(N+1)F_R^2)$, and let $W$ be the open arc of $\mathbb{R}/\mathbb{Z}$ of length $2\delta$ centred at $x_N$.
Let $x\in W$.
For $m\le R$, the point $(-1)^{m+1}F_mx$ differs from $N\varphi^{-m}$ modulo~$1$ by less than $F_m\delta\le1/(3(N+1)F_m)$, so, exactly as in step~(A) of the proof of Lemma~\ref{lem:shift}, $x\in S_m$ if and only if $x_N\in S_m$.
Since $\rows(N)\subseteq\{m:\varphi^m<N+1\}\subseteq[1,R]$, this gives $A_R(x)=a(N)=v$.
Now let $m>R$, so that $0<N\varphi^{-m}<(N+1)\varphi^{-m}<\varphi^{-3}$.
If $F_m\delta\le N\varphi^{-m}$, then modulo~$1$ the point $(-1)^{m+1}F_mx$ lies in $(0,2N\varphi^{-m})\subseteq(0,1-\varphi^{-m})$, so $x\notin S_m$.
Since $F_m\varphi^m$ increases with $m$, the other $m>R$ are exactly those with $m\ge m_0$, for some $m_0>R$, and for them the periodicity of $S_m$ gives $\lambda(W\cap S_m)\le(2\delta+1/F_m)\varphi^{-m}$.
Using $\varphi^{m-2}\le F_m\le\varphi^{m-1}$, $m_0\ge R+1\ge6$ and $N\varphi^{-m_0}<F_{m_0}\delta\le\varphi^{m_0-1}\delta$, we get
\[
\sum_{m\ge m_0}\lambda(W\cap S_m)\le2\delta\varphi^{2-m_0}+\varphi^{3-2m_0}<2\varphi^{-3}\delta+\varphi^{2}\delta/N.
\]
As $A=v$ on $W\setminus\bigcup_{m\ge m_0}S_m$, we obtain $d_v\ge(2-2\varphi^{-3}-\varphi^2/N)\delta>0$, because $N\ge2$ and $2\varphi^{-3}+\varphi^2/2<1.79$.
\end{proof}

Numerically, $d_1\approx0.47263$, $d_2\approx0.44632$, $d_3\approx0.07271$ and $d_4\approx0.00724$, close to the proportions in Table~\ref{tab:dens}.
The data also suggest the following, which we leave open.

\begin{conjecture}
For every $v\ge2$, the least $N$ with $a(N)=v$ is $L_{4v-5}$.
\end{conjecture}

By Proposition~\ref{prop:lucas}, $a(L_{4v-5})=v$, so the conjecture concerns only minimality. It holds for $v\le10$.

We also checked Theorem~\ref{thm:pairs} directly for all $1\le i<j\le12$ against the rows restricted to $[1,3\cdot10^6]$, and both parts of Corollary~\ref{cor:kimberling} for every pair of consecutive rows $(m,m+1)$ with $m\le21$, using the first $2\cdot10^4$ entries of each row.

\section*{Acknowledgements}
Clark Kimberling's OEIS entries A128440 and A358359 posed the questions answered here.
The author used an AI assistant (Claude, Anthropic) for literature searches, computations, and drafting, and checked the proofs and computations.

\end{document}